\documentclass{amsart}
\usepackage{amsfonts}

\newtheorem{theorem}{Theorem}
\theoremstyle{plain}
\newtheorem{acknowledgement}{Acknowledgement}

\newtheorem{corollary}{Corollary}

\newtheorem{lemma}{Lemma}

\newtheorem{proposition}{Proposition}
\newtheorem{remark}{Remark}

\numberwithin{equation}{section}
\input{tcilatex}

\begin{document}
\title[Askey--Wilson Integral]{Askey--Wilson Integral and its
Generalizations }
\author{Pawe\l\ J. Szab\l owski}
\address{Department of Mathematics and Information Sciences,\\
Warsaw University of Technology\\
pl. Politechniki 1, 00-661 Warsaw, Poland}
\email{pawel.szablowski@gmail.com}
\date{December 20, 2011}

\begin{abstract}
We expand the Askey--Wilson (AW) density in a series of products of
continuous $q-$Hermite polynomials times the density that makes these
polynomials orthogonal. As a by-product we obtain the value of the AW
integral as well as the values of integrals of $q-$Hermite polynomial times
the AW density ($q-$Hermite moments of AW density). Our approach uses nice,
old formulae of Carlitz and is general enough to venture a generalization.
We prove that it is possible and pave the way how to do it. As a result we
obtain system of recurrences that if solved successfully gives a sequence of
generalized AW densities with more and more parameters.
\end{abstract}

\maketitle

\section{Introduction and Preliminaries}

\subsection{Introduction}

We consider sequence of nonnegative, integrable functions: $%
g_{n}:[-1,1]\longmapsto \mathbb{R}^{+}$ defined by the formula: 
\begin{equation*}
g_{n}\left( x|\mathbf{a}^{\left( n\right) },q\right) \allowbreak
=\allowbreak f_{h}\left( x|q\right) \prod_{j=1}^{n}\varphi _{h}\left(
x|a_{j},q\right) ,
\end{equation*}%
where $\mathbf{a}^{\left( n\right) }\allowbreak =\allowbreak (a_{1},\ldots
,a_{n}),$ functions $f_{h}$ and $\varphi _{h}$ defined by (\ref{f_h}) and (%
\ref{chH}) denote in fact respectively the density of measure that makes the
so called continuous $q-$Hermite polynomials orthogonal and the
characteristic function of these polynomials calculated at points $a_{j},$ $%
j=1,\ldots ,n.$ Naturally functions $g_{n}$ are symmetric with respect to
vectors $\mathbf{a}^{\left( n\right) }.$

Our elementary but crucial for this paper observation is that examples of
such functions are proportional to the densities of measures that make
orthogonal respectively the so called continuous $q-$Hermite (q-Hermite, $%
n\allowbreak =\allowbreak 0,$ \cite{KLS}, (14.26.2)), big $q-$Hermite (bqH, $%
n\allowbreak =\allowbreak 1,$ \cite{KLS}, (14.18.2)$),$ Al-Salam--Chihara
(ASC, $n\allowbreak =\allowbreak 2,$ \cite{KLS}, (14.8.2)), continuous dual
Hahn (C2H, $n\allowbreak =\allowbreak 3$, \cite{KLS}, (14.3.2))$,$
Askey--Wilson (AW, $\allowbreak n=\allowbreak 4$, \cite{KLS}, (14.1.2))
polynomials. This observation makes functions $g_{n}$ important and what is
more exciting allows possible generalization of both AW integral as well as
AW polynomials, i.e. go beyond $n\allowbreak =\allowbreak 4.$

Similar observations were made in fact in \cite{Andrews1999} when commenting
on formula (10.11.19). Hence one can say that we are developing certain idea
of \cite{Andrews1999}.

Let us notice that this is a second attempt to generalize AW polynomials.
The other one was made in \cite{Szabl-peculiar} by generalizing certain
properties of generating functions of $q-$Hermite, bqH, ASC, C2H and AW
polynomials.

On the other hand by the observation that these functions are symmetric in
variables $\mathbf{a}^{\left( n\right) }$ we enter the fascinating world of
symmetric functions.

The paper is organized as follows. Next Subsection \ref{prelim} presents
used notation and basic families of orthogonal polynomials that will appear
in the sequel. We also present here important properties of these
polynomials.

Section \ref{main} is devoted to expanding functions $g_{n}$ in the series
of the form: 
\begin{equation*}
g_{n}\left( x|\mathbf{a}^{\left( n\right) },q\right) =A_{n}\left( \mathbf{a}%
^{\left( n\right) },q\right) f_{h}\left( x|q\right) \sum_{j\geq 0}\frac{%
T_{j}^{\left( n\right) }\left( \mathbf{a}^{\left( n\right) },q\right) }{%
\left( q\right) _{j}}h_{j}\left( x|q\right) ,
\end{equation*}%
where $\left\{ h_{n}\right\} $ denote q-Hermite polynomials, $\left\{
T_{j}^{\left( n\right) }\right\} $ are sequences of certain symmetric
functions and finally $\left\{ A_{n}\right\} $ are the values of the
integrals%
\begin{equation*}
\int_{-1}^{1}g_{n}\left( x|\mathbf{a}^{\left( n\right) },q\right) dx,
\end{equation*}%
and symbol $\left( q\right) _{j}$ is explained at the beginning of next
Subsection.

We do this effectively for $n=0,\ldots 4,$ obtaining known results in a new
way. In Section \ref{gen} we show that defined above sequences do exist and
present the way how to obtain them recursively. We are unable however to
present nice compact forms of these sequences resembling those obtained for $%
n\leq 4,$ thus posing several open questions (see Subsection \ref{open}) and
leaving the field to younger and more talented researchers.

The partially legible although not very compact form was obtained for $%
\int_{-1}^{1}g_{5}\left( x|\mathbf{a}^{\left( 5\right) },q\right) dx$ (see (%
\ref{_5})).

For $q\allowbreak =\allowbreak 0$, the case important for the rapidly
developing so called 'free probability', we give simple, compact form for $%
\int_{-1}^{1}g_{5}\left( x|\mathbf{a}^{\left( 5\right) },0\right) dx$ (see
Theorem \ref{free}, ii)) paving the way to conjecture the compact form of (%
\ref{_5}).

Tedious proofs are shifted to Section \ref{dowody}.

\subsection{Preliminaries}

\label{prelim}$q$ is a parameter. We will assume that $-1<q\leq 1$ unless
otherwise stated. Let us define $\left[ 0\right] _{q}\allowbreak
=\allowbreak 0,$ $\left[ n\right] _{q}\allowbreak =\allowbreak 1+q+\ldots
+q^{n-1}\allowbreak ,$ $\left[ n\right] _{q}!\allowbreak =\allowbreak
\prod_{j=1}^{n}\left[ j\right] _{q},$ with $\left[ 0\right] _{q}!\allowbreak
=1$ and%
\begin{equation*}
\QATOPD[ ] {n}{k}_{q}\allowbreak =\allowbreak \left\{ 
\begin{array}{ccc}
\frac{\left[ n\right] _{q}!}{\left[ n-k\right] _{q}!\left[ k\right] _{q}!} & 
, & n\geq k\geq 0 \\ 
0 & , & otherwise%
\end{array}%
\right. .
\end{equation*}%
We will use the so called $q-$Pochhammer symbol for $n\geq 1:$%
\begin{eqnarray*}
\left( a;q\right) _{n} &=&\prod_{j=0}^{n-1}\left( 1-aq^{j}\right) , \\
\left( a_{1},a_{2},\ldots ,a_{k};q\right) _{n}\allowbreak &=&\allowbreak
\prod_{j=1}^{n}\left( a_{j};q\right) _{n}.
\end{eqnarray*}%
with $\left( a;q\right) _{0}=1$.

Often $\left( a;q\right) _{n}$ as well as $\left( a_{1},a_{2},\ldots
,a_{k};q\right) _{n}$ will be abbreviated to $\left( a\right) _{n}$ and $%
\left( a_{1},a_{2},\ldots ,a_{k}\right) _{n},$ if it will not cause
misunderstanding.

It is easy to notice that $\left( q\right) _{n}=\left( 1-q\right) ^{n}\left[
n\right] _{q}!$ and that%
\begin{equation}
\QATOPD[ ] {n}{k}_{q}\allowbreak =\allowbreak \allowbreak \left\{ 
\begin{array}{ccc}
\frac{\left( q\right) _{n}}{\left( q\right) _{n-k}\left( q\right) _{k}} & ,
& n\geq k\geq 0 \\ 
0 & , & otherwise%
\end{array}%
\right. .  \label{qBin}
\end{equation}

The case $q=1$ will be considered only when it might make sense and will be
understood as the limit $q\longrightarrow 1^{-}.$

\begin{remark}
\label{particular}Notice that $\left[ n\right] _{1}\allowbreak =\allowbreak
n,\left[ n\right] _{1}!\allowbreak =\allowbreak n!,$ $\QATOPD[ ] {n}{k}%
_{1}\allowbreak =\allowbreak \binom{n}{k},$ $\left( a;1\right)
_{n}\allowbreak =\allowbreak \left( 1-a\right) ^{n}$ and $\left[ n\right]
_{0}\allowbreak =\allowbreak \left\{ 
\begin{array}{ccc}
1 & if & n\geq 1 \\ 
0 & if & n=0%
\end{array}%
\right. ,$ $\left[ n\right] _{0}!\allowbreak =\allowbreak 1,$ $\QATOPD[ ] {n%
}{k}_{0}\allowbreak =\allowbreak 1,$ for $0\leq k\leq n,$ $\left( a;0\right)
_{n}\allowbreak =\allowbreak \left\{ 
\begin{array}{ccc}
1 & if & n=0 \\ 
1-a & if & n\geq 1%
\end{array}%
\right. .$
\end{remark}

We will need the following sets of polynomials

The Rogers--Szeg\"{o} polynomials that are defined by the equality: 
\begin{equation}
w_{n}\left( x|q\right) \allowbreak =\allowbreak \sum_{k=0}^{n}\QATOPD[ ] {n}{%
k}_{q}x^{k},  \label{RS}
\end{equation}%
for $n\geq 0$ and $w_{-1}\left( x|q\right) \allowbreak =\allowbreak 0.$ They
will play an auxiliary r\^{o}le in the sequel.

In particular one shows (see e.g. \cite{IA}) that the polynomials defined
by: 
\begin{equation}
h_{n}\left( x|q\right) \allowbreak =\allowbreak e^{in\theta }w_{n}\left(
e^{-2i\theta }|q\right)  \label{cH}
\end{equation}%
where $x\allowbreak =\allowbreak \cos \theta ,$ satisfy the following 3-term
recurrence: 
\begin{equation}
h_{n+1}(x|q)=2xh_{n}(x|q)-(1-q^{n})h_{n-1}(x|q),  \label{q-H}
\end{equation}%
with $h_{-1}\left( x|q\right) \allowbreak =\allowbreak 0,$ $h_{0}\left(
x|q\right) \allowbreak =\allowbreak 1.$

These polynomials are the so called continuous $q-$Hermite polynomials. A
lot is known about their properties. For good reference see \cite{IA}, \cite%
{KLS} or \cite{Szab-rev}. In particular we know that for $\left\vert
q\right\vert <1:$ 
\begin{equation*}
\sup_{\left\vert x\right\vert \leq 1}\left\vert h_{n}\left( x|q\right)
\right\vert \leq w_{n}\left( 1|q\right) .
\end{equation*}

\begin{remark}
\label{special1}Notice that $h_{n}\left( x|0\right) \allowbreak \ $equals to 
$n-th$ Chebyshev polynomial of the second kind. More about these polynomials
one can find in e.g. \cite{Andrews1999}. To analyze the case $q\allowbreak
=\allowbreak 1$ let us consider rescaled polynomials $h_{n}$ i.e. $%
H_{n}\left( x|q\right) \allowbreak =\allowbreak h_{n}\left( x\sqrt{1-q}%
/2|q\right) /\left( 1-q\right) ^{n/2}.$ Then equation (\ref{q-H}) takes a
form:%
\begin{equation*}
H_{n+1}\left( x|q\right) \allowbreak =\allowbreak xH_{n}(x|q)-\left[ n\right]
_{q}H_{n-1}\left( x|q\right) ,
\end{equation*}%
which shows that $H_{n}\left( x|q\right) \allowbreak =\allowbreak H_{n}(x),$
where $\left\{ H_{n}\right\} $ denote the so called 'probabilistic' Hermite
polynomials i.e. polynomials orthogonal with respect to the measure with
density equal to $\exp \left( -x^{2}/2\right) /\sqrt{2\pi }.$ This
observation suggests that although the case $q\allowbreak =\allowbreak 1$
lies within our interest it requires special approach. In fact it will be
solved completely in Section \ref{gen}. For now we will assume that $%
\left\vert q\right\vert <1.$
\end{remark}

In the sequel the following identities discovered by Carlitz (see Exercise
12.3(b) and 12.3(c) of \cite{IA}), true for $\left\vert q\right\vert
,\left\vert t\right\vert <1$ : 
\begin{equation}
\sum_{k=0}^{\infty }\frac{w_{k}\left( 1|q\right) t^{k}}{\left( q\right) _{k}}%
\allowbreak =\allowbreak \frac{1}{\left( t\right) _{\infty }^{2}}%
,\sum_{k=0}^{\infty }\frac{w_{k}^{2}\left( 1|q\right) t^{k}}{\left( q\right)
_{k}}\allowbreak =\allowbreak \frac{\left( t^{2}\right) _{\infty }}{\left(
t\right) _{\infty }^{4}},  \label{Car_id}
\end{equation}%
will enable to show convergence of many series considered in the sequel.

We have also the following the so called 'linearization formula' (\cite{IA},
13.1.25) which can be dated back in fact to Rogers and Carlitz (see \cite%
{Andrews1999}, 10.11.10 with $\beta \allowbreak =\allowbreak 0$ or \cite%
{Carlitz56} for Rogers--Szeg\"{o} polynomials):%
\begin{equation}
h_{n}\left( x|q\right) h_{m}\left( x|q\right) =\sum_{j=0}^{\min \left(
n,m\right) }\QATOPD[ ] {m}{j}_{q}\QATOPD[ ] {n}{j}_{q}\left( q\right)
_{j}h_{n+m-2k}\left( x|q\right) ,  \label{linH}
\end{equation}%
that will be our basic tool.

We will use the following two formulae of Carlitz presented in \cite%
{Carlitz72}, that concern properties of Rogers--Szeg\"{o} polynomials. Let
us define two sets of functions 
\begin{eqnarray*}
\zeta _{n}\left( x|a,q\right) \allowbreak &=&\allowbreak \sum_{m\geq 0}\frac{%
a^{m}}{\left( q\right) _{m}}w_{n+m}\left( x|q\right) , \\
\lambda _{n,m}\left( x,y|a,q\right) \allowbreak &=&\allowbreak \sum_{k\geq 0}%
\frac{a^{k}}{\left( q\right) _{k}}w_{n+k}\left( x|q\right) w_{m+k}\left(
y|q\right) ,
\end{eqnarray*}%
defined for $\left\vert x\right\vert ,\left\vert y\right\vert \leq 1$, $%
\left\vert a\right\vert <1\allowbreak $ and $n,m$ being nonnegative
integers. Carlitz proved (\cite{Carlitz72}, (3.2), after correcting an
obvious misprint) that%
\begin{eqnarray}
\zeta _{n}\left( x|a,q\right) \allowbreak &=&\allowbreak \zeta _{0}\left(
x|a,q\right) \mu _{n}\left( x|a,q\right) ,  \label{Car11} \\
\zeta _{0}\left( x|a,q\right) &=&\frac{1}{\left( a,ax\right) _{\infty }},
\label{Car12}
\end{eqnarray}%
where functions $\mu _{n}$ are polynomials that are defined by: 
\begin{equation}
\mu _{n}\left( x|a,q\right) \allowbreak =\allowbreak \sum_{j=0}^{n}\QATOPD[ ]
{n}{j}_{q}\left( a\right) _{j}x^{j},  \label{mi}
\end{equation}%
and that (\cite{Carlitz72}, (1.4), case $m\allowbreak =\allowbreak 0$ also
given in \cite{IA}, Ex 12.3 (d))%
\begin{equation}
\frac{\lambda _{m,n}\left( x,y|a,q\right) }{\lambda _{0,0}\left(
x,y|a,q\right) }=\sum_{j=0}^{m}\sum_{k=0}^{n}\QATOPD[ ] {n}{k}_{q}\QATOPD[ ]
{m}{j}_{q}\frac{\left( ax\right) _{j}\left( ay\right) _{k}\left( xya\right)
_{k+j}}{\left( xya^{2}\right) _{k+j}}x^{m-j}y^{n-k},  \label{Car21}
\end{equation}%
with 
\begin{equation}
\lambda _{0,0}\left( x,y|a,q\right) =\frac{\left( xya^{2}\right) _{\infty }}{%
\left( a,ax,ay,axy\right) _{\infty }}.  \label{Car22}
\end{equation}

It is elementary to prove the following two properties of the polynomials $%
\mu _{n},$ hence we present them without the proof.

\begin{proposition}
\begin{eqnarray}
x^{n}\mu _{n}\left( x^{-1}|a,q\right) &=&\sum_{j=0}^{n}\QATOPD[ ] {n}{j}%
_{q}(-a)^{j}q^{\binom{j}{2}}w_{n-j}\left( x|q\right)  \label{symp1} \\
w_{n}\left( x|q\right) \allowbreak &=&\allowbreak \sum_{k=0}^{n}\QATOPD[ ] {n%
}{k}_{q}a^{k}x^{n-k}\mu _{n-k}\left( x^{-1}|a,q\right) .  \label{symp2}
\end{eqnarray}
\end{proposition}

To perform our calculations we will need also the following two functions.

The generating function of the $q-$Hermite polynomials that is given by the
formula below (see \cite{KLS}, (14.26.1)):%
\begin{equation}
\varphi _{h}\left( x|t,q\right) \overset{df}{=}\sum_{j\geq 0}\frac{t^{j}}{%
\left( q\right) _{j}}h_{j}\left( x|q\right) \allowbreak =\allowbreak \frac{1%
}{\prod_{k=0}^{\infty }v\left( x|tq^{k}\right) },  \label{chH}
\end{equation}%
where $v\left( x|t\right) \allowbreak =\allowbreak 1-2tx+t^{2}.$ Notice that 
$v\left( x|t\right) \geq 0$ for $\left\vert x\right\vert \leq 1$ and that
from (\ref{Car_id}) it follows that series in (\ref{cH}) converges for $%
\left\vert t\right\vert <1.$ Notice also that from (\ref{Car_id}) it follows
that:%
\begin{equation}
\sup_{\left\vert x\right\vert \leq 1}\varphi _{h}\left( x|t,q\right)
\allowbreak =\allowbreak 1/\left( \left\vert t\right\vert \right) _{\infty
}^{2}.  \label{ogr_fi}
\end{equation}

The density of the measure with respect to which polynomials $h_{n}$ are
orthogonal is given in e.g. \cite{KLS}, (14.26.2). Following it we have%
\begin{equation*}
\int_{-1}^{1}h_{n}\left( x|q\right) h_{m}\left( x|q\right) f_{h}\left(
x|q\right) dx=\left( q\right) _{n}\delta _{nm},
\end{equation*}%
where $\delta _{mn}$ denotes Kronecker's delta, and 
\begin{equation}
f_{h}\left( x|q\right) \allowbreak =\allowbreak \frac{2\left( q\right)
_{\infty }\sqrt{1-x^{2}}}{\pi }\prod_{k=1}^{\infty }l\left( x|q^{i}\right) ,
\label{f_h}
\end{equation}%
where $l\left( x|a\right) \allowbreak =\allowbreak \left( 1+a\right)
^{2}-4ax^{2}.$ Notice that 
\begin{equation}
\sup_{\left\vert x\right\vert \leq 1}f_{h}(x,q)\leq 2\left( q\right)
_{\infty }\left( -q\right) _{\infty }^{2}/\pi ,  \label{ogrfh}
\end{equation}%
following (\ref{f_h}) since $l\left( x|q\right) \leq (1+q)^{2}$ for $%
\left\vert x\right\vert \leq 1.$

\begin{remark}
\label{special}We have 
\begin{equation*}
f_{h}\left( x|0\right) =2\sqrt{1-x^{2}}/\pi ,~~\varphi _{h}\left(
x|a,0\right) =1/\left( 1-2ax+a^{2}\right) ,
\end{equation*}%
for $\left\vert x\right\vert ,\left\vert a\right\vert <1.$

After proper rescaling and normalization similar to the one performed in
Remark \ref{special1}, the case $q\allowbreak =\allowbreak 1$ leads to:%
\begin{equation*}
\exp \left( -x^{2}/2\right) /\sqrt{2\pi },~~\exp \left( ax-a^{2}/2\right) ,
\end{equation*}%
for $x,a\in \mathbb{R},$ as respectively the density of orthogonalizing
measure and the characteristic function. For details see \cite{bms} or \cite%
{SzablAW}.
\end{remark}

\section{Main results}

\label{main}Since in our approach symmetric polynomials will appear let us
introduce the following set of symmetric polynomials of $k$ variables:%
\begin{equation}
S_{n}^{\left( k\right) }\left( a_{1},\ldots ,a_{k}|q\right) =\sum_{\substack{
j_{1},\ldots ,j_{k-1}\geq 0  \\ j_{1}+\ldots +j_{k-1}\leq n}}[j_{1},\ldots
,n-\sum_{m=1}^{k-1}j_{m}]_{q}a_{1}^{j_{1}}\ldots
a_{k-1}^{j_{k-1}}a_{k}^{n-j_{1}-\ldots -j_{k-1}}.  \label{defS}
\end{equation}

where we denoted by $[j_{1},j_{2},\ldots ,n-\sum_{m=1}^{k-1}j_{m}]_{q}$ the
so called $q-$multinomial coefficient defined by $[n_{1},\ldots
,n_{m}]_{q}\allowbreak =\allowbreak \left( q\right) _{n_{1}+\ldots
+n_{m}}/\prod_{k=1}^{m}\left( q\right) _{n_{k}}.$

\begin{remark}
Notice that $S_{n}^{\left( k\right) }\left( a_{1},\ldots ,a_{k}|1\right)
\allowbreak =\allowbreak \left( \sum_{j=1}^{k}a_{j}\right) ^{n}.$
\end{remark}

\begin{proof}
Obvious since $\left. \frac{\left( q\right) _{n}}{\prod_{m=0}^{k-1}\left(
q\right) _{jm}\left( q\right) _{n-j_{1}-\ldots -j_{k-1}}}\right\vert
_{q=1}\allowbreak =\allowbreak \frac{n!}{(n-\sum_{m=1}^{k-1}j_{m})!%
\prod_{m=1}^{k-1}j_{m}!}.$
\end{proof}

\begin{proposition}
\label{ch_sym}Let $q\in \left( -1,1\right) $ then i) 
\begin{equation}
\sum_{n\geq 0}\frac{t^{n}}{\left( q\right) _{n}}S_{n}^{\left( k\right)
}\left( a_{1},\ldots ,a_{k}|q\right) \allowbreak =\allowbreak \frac{1}{%
\prod_{j=1}^{k}\left( a_{i}t\right) _{\infty }},  \label{chS}
\end{equation}

ii) for $\left\vert t\right\vert <1$ and $\forall j=1,\ldots ,k$ 
\begin{equation}
S_{n}^{\left( k\right) }\left( a_{1},\ldots ,a_{k}|q\right) =\sum_{m=0}^{n}%
\QATOPD[ ] {n}{m}_{q}S_{m}^{\left( j\right) }(a_{1},\ldots
,a_{j})S_{n-m}^{\left( k-j\right) }\left( a_{j+1},\ldots ,a_{k}|q\right) ,
\label{rozklS}
\end{equation}

If $q=1,$ then 
\begin{equation*}
\sum_{n\geq 0}\frac{t^{n}}{n!}S_{n}^{\left( k\right) }\left( a_{1},\ldots
,a_{k}|1\right) \allowbreak =\allowbreak \exp \left(
t\sum_{j=0}^{k}a_{j}\right) .
\end{equation*}

iii) 
\begin{equation*}
\left\vert S_{n}^{\left( k\right) }\left( a_{1},\ldots ,a_{k}|q\right)
\right\vert \leq \left( \max_{1\leq j\leq k}\left\vert a_{j}\right\vert
\right) ^{n}S_{n}^{\left( k\right) }\left( 1,\ldots ,1|q\right) .
\end{equation*}
\end{proposition}

\begin{proof}
i) Notice that%
\begin{equation*}
\sum_{n\geq 0}\frac{t^{n}}{\left( q\right) _{n}}S_{n}^{\left( k\right)
}\left( a_{1},\ldots ,a_{k}|q\right) \allowbreak =\allowbreak \sum_{n\geq
0}\sum_{\substack{ j_{1},\ldots ,j_{k-1}\geq 0  \\ j_{1}+\ldots +j_{k-1}\leq
n }}\frac{(ta_{1})^{j_{1}}\ldots
(ta_{k-1})^{j_{k-1}}(ta_{k})^{n-j_{1}-\ldots -j_{k-1}}}{\prod_{m=0}^{k-1}%
\left( q\right) _{jm}\left( q\right) _{n-j_{1}-\ldots -j_{k}}}.
\end{equation*}%
Secondly recall that $\frac{1}{\left( a\right) _{\infty }}\allowbreak
=\allowbreak \sum_{j\geq 0}\frac{a^{j}}{\left( q\right) _{j}}.$ Now the
assertion is easy. ii) follows either direct calculation or i) and the
properties of characteristic functions. iii) We use (\ref{defS}).
\end{proof}

Recall (i.e. \cite{IA} or \cite{KLS}) that there exist sets of orthogonal
polynomials forming a part of the so called 'AW scheme' that are orthogonal
with respect to measures with densities mentioned below. Although our main
interest is in providing simple proof of the so called AW integral we will
list related densities for better exposition and for indicating the ways of
possible generalization of AW integrals and polynomials.

So let us mention first the so called big $q-$Hermite polynomials $\left\{
h_{n}\left( x|a,q\right) \right\} _{n\geq -1}$ whose orthogonalizing measure
has density for $\left\vert a\right\vert <1$. The density $f_{bh}$ of the
orthogonalizing measure has a form (see \cite{KLS}, (14.18.2)) which can be
written with the help of functions $f_{h}$ and $\varphi _{h}.$ Namely: 
\begin{eqnarray}
f_{bh}\left( x|a,q\right) \allowbreak &=&\allowbreak f_{h}\left( x|q\right)
\varphi _{h}\left( x|a,q\right) ,  \label{fbh} \\
\int_{-1}^{1}h_{n}\left( x|a,q\right) h_{m}\left( x|a,q\right) f_{bh}\left(
x|a,q\right) &=&\left( q\right) _{n}\delta _{mn}.  \label{o_bh}
\end{eqnarray}%
The form of polynomials $h_{n}\left( x|a,q\right) $ and their relation to $%
q- $Hermite polynomials is not important for our considerations. It can be
found e.g. in \cite{KLS}, (14.26.1) or in \cite{Szab-bAW} , (2.11, 2.12). So
for the sake of completeness let us remark that from (\ref{fbh}) it follows
immediately that for $\left\vert x\right\vert \leq 1,$ $\left\vert
a\right\vert <1$ 
\begin{equation*}
f_{bh}\left( x|a,q\right) \allowbreak =\allowbreak f_{h}\left( x|q\right)
\sum_{n\geq 0}\frac{a^{n}}{\left( q\right) _{n}}h_{n}\left( x|q\right) .
\end{equation*}%
Here and below, where we will present similar expansions convergence is
almost uniform since all these expansions are in fact the Fourier series and
that the Rademacher--Menshov theorem can be applied following (\ref{Car_id}).

Let us notice immediately that following (\ref{fbh}) we have:%
\begin{equation*}
\int_{-1}^{1}h_{n}\left( x|q\right) f_{bh}\left( x|a,q\right) dx=a^{n}.
\end{equation*}

Secondly let us mention the so called Al-Salam--Chihara polynomials $\left\{
Q_{n}\left( x|a,b,q\right) \right\} _{n\geq -1}$ that are orthogonal with
respect to the measure that for $\left\vert a\right\vert ,\left\vert
b\right\vert <1$ has the density of the form (compare \cite{KLS}, (14.8.2)) 
\begin{equation}
f_{Q}\left( x|a,b,q\right) \allowbreak =\allowbreak \left( ab\right)
_{\infty }f_{h}\left( x|q\right) \varphi _{h}\left( x|a,q\right) \varphi
_{h}\left( x|b,q\right) .  \label{fQ}
\end{equation}

We have the following Lemma that illustrates our method as well as to will
give a very simple proof of well known so called Poisson--Mehler formula as
a corollary.

\begin{lemma}
\label{expfQ}For $\left\vert x\right\vert \leq 1,$ $\left\vert a\right\vert
,\left\vert b\right\vert <1$ we have 
\begin{equation}
f_{Q}\left( x|a,b,q\right) \allowbreak =\allowbreak f_{h}\left( x|q\right)
\sum_{j=0}^{\infty }\frac{S_{j}^{(2)}\left( a,b\right) }{\left( q\right) _{j}%
}h_{j}\left( x|q\right) .  \label{rozfQ}
\end{equation}
\end{lemma}

\begin{proof}
Following (\ref{fQ}) and (\ref{chH}) we have :%
\begin{equation*}
f_{Q}\left( x|a,b,q\right) \allowbreak =\allowbreak \left( ab\right)
_{\infty }f_{h}\left( x|q\right) \sum_{j,k\geq 0}\frac{a^{j}b^{k}}{\left(
q\right) _{j}\left( q\right) _{k}}h_{j}\left( x|q\right) h_{k}\left(
x|q\right) .
\end{equation*}%
Now we use (\ref{linH}) and (\ref{qBin}) and change the order of summation
getting:%
\begin{eqnarray*}
f_{Q}\left( x|a,b,q\right) &=&\left( ab\right) _{\infty }f_{h}\left(
x|q\right) \sum_{m\geq 0}\frac{\left( ab\right) ^{m}}{\left( q\right) _{m}}%
\sum_{j,k\geq m}\frac{a^{j-m}b^{k-m}}{\left( q\right) _{j-m}\left( q\right)
_{k-m}}h_{j-k+m-k}\left( x|q\right) \\
&=&\left( ab\right) _{\infty }f_{h}\left( x|q\right) \sum_{m\geq 0}\frac{%
\left( ab\right) ^{m}}{\left( q\right) _{m}}\sum_{n,i\geq 0}\frac{a^{n}b^{i}%
}{\left( q\right) _{i}\left( q\right) _{n}}h_{n+i}\left( x|q\right) \\
&=&f_{h}\left( x|q\right) \sum_{s\geq 0}\frac{h_{s}\left( x|q\right) }{%
\left( q\right) _{s}}\sum_{n=0}^{s}\QATOPD[ ] {s}{j}_{q}a^{n}b^{s-n}.
\end{eqnarray*}
\end{proof}

As an immediate corollary of our result we have:%
\begin{equation}
\int_{-1}^{1}h_{n}\left( x|q\right) f_{Q}\left( x|a,b,q\right)
dx=S_{n}^{\left( 2\right) }\left( a,b|q\right) .  \label{intH}
\end{equation}

\begin{remark}
\label{zesp}Let $a\allowbreak =\allowbreak \rho e^{i\eta },$ $b\allowbreak
=\allowbreak \rho e^{-i\eta }$ and denote $y\allowbreak =\allowbreak \cos
\eta .$ Then

i) $S_{n}^{\left( 2\right) }\left( a,b|q\right) \allowbreak =\allowbreak
\rho ^{n}h_{n}\left( y|q\right) ,$

ii) $v\left( x|a\right) v\left( x|b\right) \allowbreak =\allowbreak \left(
1-\rho ^{2}\right) ^{2}-4xy\rho \left( 1+\rho ^{2}\right) +4\rho ^{2}\left(
x^{2}+y^{2}\right) $
\end{remark}

\begin{proof}
i) is an immediate consequence of (\ref{cH}). ii) We have $v\left(
x|a\right) v\left( x|b\right) \allowbreak =\allowbreak (1-2\rho xe^{i\eta
}+\rho ^{2}e^{2i\eta })(1-2\rho xe^{-i\eta }+\rho ^{2}e^{-2i\eta })$
\end{proof}

As a slightly more complicated corollary implied by Lemma \ref{expfQ} we
have the following famous Poisson--Mehler (PM) expansion formula:

\begin{corollary}
\label{PM}For $\left\vert x\right\vert ,\left\vert y\right\vert <1,$ $%
\left\vert \rho \right\vert <1$ we have%
\begin{eqnarray}
&&\frac{\left( \rho ^{2}\right) _{\infty }}{\prod_{k=0}^{\infty }\left(
1-\rho ^{2}q^{2k}\right) ^{2}-4xy\rho q^{k}\left( 1+\rho ^{2}q^{2k}\right)
+4\rho ^{2}q^{2k}\left( x^{2}+y^{2}\right) }  \label{pm} \\
&=&\sum_{j\geq 0}\frac{\rho ^{j}}{\left( q\right) _{j}}h_{j}\left(
x|q\right) h_{j}\left( y|q\right) .  \notag
\end{eqnarray}
\end{corollary}

\begin{proof}
We take $a\allowbreak =\allowbreak \rho e^{i\eta },$ $b\allowbreak
=\allowbreak \rho e^{-i\eta }$ and denote $y\allowbreak =\allowbreak \cos
\eta .$ Now we use (\ref{fQ}) and Remark \ref{zesp}, ii) to get left hand
side multiplied by $f_{h}.$ Then we apply Lemma \ref{expfQ}, and Remark \ref%
{zesp}, i) to get right hand side of our PM formula also multiplied by $%
f_{h} $. Finally we cancel out $f_{h}$ which is positive on $(-1,1).$
\end{proof}

\begin{remark}
The calculations we have performed while proving Lemma \ref{expfQ} are very
much like those performed in \cite{IA} while proving of Theorem 13.1.6
concerning Poisson kernel (or Poisson--Mehler) formula. There exist may
proofs of PM formula, see e.g. \cite{bressoud} or recently obtained very
short in \cite{Szab-Exp}. In fact the formula (\ref{pm}) can be dated back
to Carlitz who in \cite{Carlitz57} formulated it for Rogers--Szeg\"{o}
polynomials. The one presented above seems to be one of the shortest, was
obtained as a by-product and as already mentioned is almost the same as the
one presented in \cite{IA}.
\end{remark}

Notice that considering (\ref{intH}) with $a\allowbreak =\allowbreak \rho
e^{i\eta },$ $b\allowbreak =\allowbreak \rho e^{-i\eta }$ and $y\allowbreak
=\allowbreak \cos \eta $ leads in view of Remark \ref{zesp}, i) to 
\begin{equation*}
\int_{-1}^{1}h_{n}\left( x|q\right) f_{Q}\left( x|a,b,q\right) dx=\rho
^{n}h_{n}\left( y|q\right) ,
\end{equation*}%
a nice symmetric formula that appeared in \cite{bryc1} in probabilistic
context. Its probabilistic interpretation was exploited further in \cite{bms}%
.

Third in our sequence of families of polynomials that constitute AW scheme
are the so called continuous dual Hahn (C2H) polynomials. Again their
relationship to other sets of polynomials is not important. From \cite{KLS},
(14.3.2) it follows that the density of measure that makes them orthogonal
is given by the following formula. 
\begin{equation*}
f_{CH}\left( x|a,b,c,q\right) \allowbreak =\allowbreak \left(
ab,ac,bc\right) _{\infty }f_{h}\left( x|q\right) \varphi _{h}\left(
x|a,q\right) \varphi _{h}\left( x|b,q\right) \varphi _{h}\left( x|c,q\right)
.
\end{equation*}%
We have the following lemma.

\begin{lemma}
\label{C2H}%
\begin{equation*}
f_{CH}\left( x|a,b,c,q\right) =f_{h}\left( x|q\right) \sum_{n\geq 0}\frac{%
\sigma _{n}^{\left( 3\right) }\left( a,b,c|q\right) }{\left( q\right) _{n}}%
h_{n}\left( x|q\right) ,
\end{equation*}%
where 
\begin{equation}
\sigma _{n}^{\left( 3\right) }\left( a,b,c|q\right) =\sum_{j=0}^{n}\QATOPD[ ]
{n}{j}_{q}q^{\binom{j}{2}}\left( -abc\right) ^{j}S_{n-j}^{\left( 3\right)
}\left( a,b,c|q\right) .  \label{sigma3}
\end{equation}
\end{lemma}

\begin{proof}
Is shifted to Section \ref{dowody}.
\end{proof}

\begin{remark}
Notice that for $\left\vert t\right\vert <1$ 
\begin{equation*}
\sum_{n\geq 0}\frac{t^{n}}{\left( q\right) _{n}}\sigma _{n}^{\left( 3\right)
}\left( a,c,b|q\right) \allowbreak =\allowbreak \frac{\left( abct\right)
_{\infty }}{\left( at,bt,ct\right) _{\infty }}.
\end{equation*}
\end{remark}

\begin{proof}
Using (\ref{sigma3}) we have:%
\begin{eqnarray*}
\sum_{n\geq 0}\frac{t^{n}}{\left( q\right) _{n}}\sigma _{n}^{\left( 3\right)
}\left( a,c,b|q\right) &=&\sum_{n\geq 0}\frac{t^{n}}{\left( q\right) _{n}}%
\sum_{j=0}^{n}\QATOPD[ ] {n}{j}_{q}q^{\binom{j}{2}}\left( -abc\right)
^{j}S_{n-j}^{\left( 3\right) }\left( a,b,c|q\right) \\
&=&\sum_{j=0}^{\infty }\frac{(-abct)^{j}}{\left( q\right) _{j}}q^{\binom{j}{2%
}}\sum_{n\geq j}\frac{t^{n-j}}{\left( q\right) _{n-j}}S_{n-j}^{\left(
3\right) }\left( a,b,c|q\right) .
\end{eqnarray*}%
Now it remains to change the index of summation in the second sum, use (\ref%
{chS}) and use the fact that $\sum_{j=0}^{\infty }\frac{(-abct)^{j}}{\left(
q\right) _{j}}q^{\binom{j}{2}}\allowbreak =\allowbreak \left( abct\right)
_{\infty }.$
\end{proof}

\begin{corollary}
For $\left\vert a|,|b\right\vert ,\left\vert c\right\vert <1:$%
\begin{equation*}
\int_{-1}^{1}h_{n}\left( x|q\right) f_{CH}\left( x|a,b,c,q\right) dx=\sigma
_{n}^{\left( 3\right) }\left( a,b,c|q\right) .
\end{equation*}
\end{corollary}

\begin{proof}
Elementary.
\end{proof}

Fourth family of polynomials that constitute AW scheme are the celebrated
Askey--Wilson polynomials. Again their form and relationship to other
families of polynomials of AW scheme is not important for our
considerations. Recently a relatively rich study of these relationships was
done in \cite{Szab-bAW} hence it may serve as the reference. We need only
the form of AW density. It is given e.g. in \cite{KLS}, (14.1.2) and after
necessary adaptation to our notation is presented below:%
\begin{gather*}
f_{AW}\left( x|a,b,c,d,q\right) =\frac{\left( ab,ac,ad,bc,bd,cd\right)
_{\infty }}{\left( abcd\right) _{\infty }} \\
\times f_{h}\left( x|a\right) \varphi _{h}\left( x|a,q\right) \varphi
_{h}\left( x|b,q\right) \varphi _{h}\left( x|c,q\right) \varphi _{h}\left(
x|d,q\right) ,
\end{gather*}
for $\left\vert x\right\vert \leq 1,$ $\left\vert a\right\vert ,\left\vert
b\right\vert ,\left\vert c\right\vert ,\left\vert d\right\vert <1.$ Our main
result concerns this density and is the following:

\begin{theorem}
\label{fAW}For $\left\vert x\right\vert \leq 1,$ $\left\vert a\right\vert
,\left\vert b\right\vert ,\left\vert c\right\vert ,\left\vert d\right\vert
<1 $%
\begin{equation}
f_{AW}\left( x|a,b,c,d,q\right) =f_{h}\left( x|q\right) \sum_{n\geq 0}\frac{%
\sigma _{n}^{\left( 4\right) }\left( a,b,c,d|q\right) }{\left( q\right) _{n}}%
h_{n}\left( x|q\right) ,  \label{fAWexp}
\end{equation}%
where 
\begin{equation}
\sigma _{n}^{\left( 4\right) }\left( a,b,c,d|q\right) =\sum_{j=0}^{n}\QATOPD[
] {n}{j}_{q}\frac{\left( bd\right) _{j}}{\left( abcd\right) _{j}}%
S_{n-j}^{\left( 2\right) }\left( b,d|q\right) \sum_{k=0}^{j}\QATOPD[ ] {j}{k}%
_{q}\left( cb\right) _{k}a^{k}\left( ad\right) _{j-k}c^{j-k},  \label{sig4}
\end{equation}%
are symmetric functions of $a,$ $b,$ $c,$ $d.$
\end{theorem}

\begin{proof}
is shifted to section \ref{dowody}.
\end{proof}

As immediate corollaries we have the following fact.

\begin{corollary}
\label{AW_int}For $\max (\left\vert a\right\vert ,\left\vert b\right\vert
,\left\vert c\right\vert ,\left\vert d\right\vert )<1:$%
\begin{equation}
\int_{-1}^{1}h_{n}\left( x\right) f_{AW}\left( x|a,b,c,d,q\right) dx=\sigma
_{n}^{\left( 4\right) }\left( a,b,c,d|q\right) .  \label{AWmom}
\end{equation}
\end{corollary}

\begin{proof}
Follows directly from (\ref{fAWexp}).
\end{proof}

\begin{remark}
Notice that from (\ref{fAWexp}) follows in fact the value of AW integral,
since we see that $\int_{-1}^{1}f_{AW}\left( x|a,b,c,d|q\right) \allowbreak
=\allowbreak 1$ which means that the integral 
\begin{gather}
\frac{1}{2\pi }\int_{-1}^{1}\frac{1}{\sqrt{1-x^{2}}}\prod_{n\geq 0}\frac{%
l\left( x|q^{n}\right) }{v\left( x|aq^{n}\right) v\left( x|bq^{n}\right)
v\left( x|cq^{n}\right) v\left( x|dq^{n}\right) }dx  \label{AWint} \\
=\frac{\left( abcd\right) _{\infty }}{\left( q,ab,ac,ad,bc,bd,cd\right)
_{\infty }}.  \notag
\end{gather}%
(\ref{AWint}) is nothing else but the celebrated AW integral. Notice also
that recently there appeared at least two papers \cite{Liu2009}, \cite%
{Ma2011} where (\ref{AWint}) was derived from much more advanced theorems.
\end{remark}

\begin{remark}
Notice also that (\ref{AWmom}) allows calculation of all moments of AW
density. This is so since one knows the form of polynomials $h_{n}.$ Moments
of AW density were calculated by Corteel et. al. in 2010 in \cite{Corteel10}
using combinatorial means. For complex $a,$ $b,$ $c,$ $d$ but forming
conjugate pairs this formula was also obtained independently about the same
time. Namely it was done in \cite{SzablAW} where also an elegant expansion
of $\sigma _{n}^{\left( 4\right) }\left( \rho _{1}e^{i\eta },\rho
_{1}e^{-i\eta },\rho _{2}e^{i\theta },\rho _{2}e^{-i\theta }|q\right) $ in
terms of $h_{n}\left( y|q\right) $ and $h_{n}\left( z|q\right) ,$ where $%
\cos \eta \allowbreak =\allowbreak y$ and $\cos \theta \allowbreak
=\allowbreak z\ $was presented.
\end{remark}

\section{Generalization and open questions}

\label{gen}

\subsection{Generalization}

The presented above results allow the following generalization. The cases $%
\left\vert q\right\vert <1$ and $q\allowbreak =\allowbreak 1$ will be
treated separately. First let us consider $\left\vert q\right\vert <1$.

Let us denote $\mathbf{a}^{\left( k\right) }\allowbreak =\allowbreak \left(
a_{1},\ldots ,a_{k}\right) ,$ $k\allowbreak =\allowbreak 0,1,\ldots $ . We
will assume that $\left\vert x\right\vert \leq 1$ and that all parameters $%
a_{i}$ have absolute values less that $1.$ Let us denote 
\begin{equation*}
g_{n}\left( x|\mathbf{a}^{\left( n\right) },q\right) \allowbreak
=\allowbreak f_{h}\left( x|q\right) \prod_{j=1}^{n}\varphi _{h}\left(
x|a_{i},q\right) ,
\end{equation*}%
where functions $f_{h}$ and $\varphi _{h}$ were defined by (\ref{f_h}) and (%
\ref{chH}) respectively.

We remark following (\ref{ogr_fi}) and (\ref{ogrfh}) that 
\begin{equation}
g_{n}\left( x|\mathbf{a}^{\left( n\right) },q\right) \leq \frac{2(q)_{\infty
}\left( -q\right) _{\infty }^{2}}{\pi }\prod_{j=1}^{n}\frac{1}{\left(
\left\vert a_{j}\right\vert \right) _{\infty }^{2}},  \label{ogr_g}
\end{equation}
for $\left\vert x\right\vert \leq 1,$ and $\left\vert a_{i}\right\vert <1$
for $j\allowbreak =\allowbreak 1,\ldots ,n.$

We have the following general result.

\begin{lemma}
For every $n\geq 0$ $,$ there exist $A_{n}\left( \mathbf{a}^{\left( n\right)
},q\right) $ a symmetric function of $\mathbf{a}^{\left( n\right) }$ and a
sequence of symmetric in $\mathbf{a}^{\left( n\right) }$ functions $\left\{
T_{j}^{\left( n\right) }\left( \mathbf{a}^{\left( n\right) },q\right)
\right\} _{j\geq 0}$ such that for $\left\vert a_{k}\right\vert <1,$ $%
k\allowbreak =\allowbreak 1,\ldots ,n:$ 
\begin{equation}
g_{n}\left( x|\mathbf{a}^{\left( n\right) },q\right) =A_{n}\left( \mathbf{a}%
^{\left( n\right) },q\right) f_{h}\left( x|q\right) \sum_{j\geq 0}\frac{%
T_{j}^{\left( n\right) }\left( \mathbf{a}^{\left( n\right) },q\right) }{%
\left( q\right) _{j}}h_{j}\left( x|q\right) .  \label{symroz}
\end{equation}%
Moreover 
\begin{equation}
\sum_{j\geq 0}\left( T_{j}^{\left( n\right) }\left( \mathbf{a}^{\left(
n\right) },q\right) \right) ^{2}<\infty .  \label{parseval}
\end{equation}
\end{lemma}

\begin{proof}
Let $\mathcal{G\allowbreak =\allowbreak }L_{2}\left( <-1,1>,\mathcal{F}%
,f_{h}\right) $ be the space of functions $h:<-1,1>\allowbreak \longmapsto 
\mathbb{R}$ such that $\int_{-1}^{1}h^{2}\left( x\right) f_{h}\left(
x|q\right) dx.$ Notice that this space is spanned by the polynomials $%
\left\{ h_{j}\left( x|q\right) \right\} _{j\geq 0}.$ Visibly, under our
assumptions and by (\ref{ogr_fi}), $\prod_{j=1}^{n}\varphi _{h}\left(
x|a_{i},q\right) \allowbreak \in \allowbreak \mathcal{G}.$ Now notice that $%
\left\{ T_{j}^{\left( n\right) }\left( \mathbf{a}^{\left( n\right)
},q\right) \right\} _{j\geq 0}$ are coefficients of the Fourier expansion of
the function $\prod_{j=1}^{n}\varphi _{h}\left( x|a_{i},q\right) $ in $%
\mathcal{G}$ with respect to $\left\{ h_{j}\left( x|q\right) \right\}
_{j\geq 0}.$ Since 
\begin{equation*}
\int_{-1}^{1}f_{h}\left( x|q\right) \sum_{j\geq 0}\frac{T_{j}^{\left(
n\right) }\left( \mathbf{a}^{\left( n\right) },q\right) }{\left( q\right)
_{j}}h_{j}\left( x|q\right) dx\allowbreak =\allowbreak 1,
\end{equation*}%
$A_{n}\left( \mathbf{a}^{\left( n\right) },q\right) $ is the value of $%
\int_{-1}^{1}g_{n}\left( x|\mathbf{a}^{\left( n\right) },q\right) dx.$ (\ref%
{parseval}) follow properties of the Fourier expansion more precise the
Perseval's identity. The fact that $A_{n}$ and $\left\{ T_{j}^{\left(
n\right) }\right\} _{j\geq 0}$ are symmetric follows the observations that $%
\prod_{j=1}^{n}\varphi _{h}\left( x|a_{i},q\right) $ is symmetric.
\end{proof}

Using formula (\ref{mi}) we can write $g_{n}$ in the following way where $%
h_{j}$ are $q-$Hermite polynomials defined by (\ref{q-H}). Functions $%
A_{n}\left( \mathbf{a}^{\left( n\right) },q\right) $ and $\left\{
T_{j}^{\left( n\right) }\left( \mathbf{a}^{\left( n\right) },q\right)
\right\} _{j\geq 0}$ have the following interpretation:%
\begin{equation*}
\int_{\lbrack -1,1]}h_{j}\left( x|q\right) g_{n}\left( x|\mathbf{a}^{\left(
n\right) },q\right) dx=A_{n}\left( \mathbf{a}^{\left( n\right) },q\right)
T_{j}^{\left( n\right) }\left( \mathbf{a}^{\left( n\right) },q\right) ,
\end{equation*}%
for $n,j\geq 0.$

We have the following easy Proposition giving recursions that are satisfied
by functions $A_{n}$ and $T_{j}^{\left( n\right) }.$

\begin{proposition}
\label{GEN}Let us define new sequence of functions $\left\{ H_{s}\left( 
\mathbf{a}^{\left( n\right) },q\right) \right\} _{n,s\geq 0}$ of $n$
variables:%
\begin{equation*}
\sum_{m\geq 0}\frac{a_{n}^{m}}{\left( q\right) _{m}}T_{s+m}^{\left(
n-1\right) }\left( \mathbf{a}^{\left( n-1\right) },q\right) =H_{s}^{\left(
n\right) }\left( \mathbf{a}^{\left( n\right) },q\right) \sum_{m\geq 0}\frac{%
a_{n}^{m}}{\left( q\right) _{m}}T_{m}^{\left( n-1\right) }\left( \mathbf{a}%
^{\left( n-1\right) },q\right) .
\end{equation*}

Then i)%
\begin{equation*}
A_{n}\left( \mathbf{a}^{\left( n\right) },q\right) =A_{n-1}\left( \mathbf{a}%
^{\left( n-1\right) },q\right) \sum_{m\geq 0}\frac{a_{n}^{m}}{\left(
q\right) _{m}}T_{m}^{\left( n-1\right) }\left( \mathbf{a}^{\left( n-1\right)
},q\right) ,
\end{equation*}

ii) 
\begin{equation*}
T_{j}^{\left( n\right) }\left( \mathbf{a}^{\left( n\right) },q\right)
=\sum_{s=0}^{j}\QATOPD[ ] {j}{s}_{q}H_{s}^{\left( n-1\right) }\left( \mathbf{%
a}^{\left( n-1\right) },q\right) \left( a_{n}\right) ^{j-s}.
\end{equation*}
\end{proposition}

\begin{proof}
Proof is shifted to section \ref{dowody}.
\end{proof}

\begin{remark}
The integral $\int_{-1}^{1}g_{n}\left( x|\mathbf{a}^{\left( n\right)
},q\right) dx$ has been calculated in \cite{ISV87} (see also theorem 15.3.1
in \cite{IA}) by combinatorial methods. Obtained formula is however very
complicated. Besides above mentioned Theorem 15.3.1 of \cite{IA} does not
provide expansion (\ref{symroz}) which is automatically obtained in our
approach.
\end{remark}

\begin{remark}
Notice also that following Proposition \ref{GEN}, i) we get for $\left\vert
a_{j}\right\vert <1,$ $j\allowbreak =\allowbreak 1,\ldots ,5:$ 
\begin{equation}
\int_{-1}^{1}g_{5}\left( x|\mathbf{a}^{\left( 5\right) },q\right)
\allowbreak =\allowbreak \frac{\left( \prod_{j}^{4}a_{j}\right) _{\infty }}{%
(q)_{\infty }\prod_{1\leq k<m\leq 4}\left( a_{k}a_{m}\right) _{\infty }}%
\sum_{j\geq 0}\frac{a_{5}^{j}}{\left( q\right) _{j}}\sigma _{j}^{\left(
4\right) }\left( a_{1},a_{2},a_{3},a_{4}|q\right)  \label{_5}
\end{equation}
\end{remark}

For $q\allowbreak =\allowbreak 0$ the calculations presented in (\ref{_5})
can be carried out completely and the concise form can be obtained. This is
possible due to the following simplified form of (\ref{sig4}).

\begin{theorem}
\label{free}Let $\mathbf{a}^{\left( 5\right) }\allowbreak .$ Under $%
\left\vert a_{j}\right\vert <1,$ $j\allowbreak =\allowbreak 1,\ldots ,5$ we
have:

i) 
\begin{gather*}
\sigma _{n}^{\left( 4\right) }\allowbreak
(a_{1},a_{2},a_{3},a_{4}|0)=\allowbreak S_{n}^{\left( 2\right)
}(a_{2},a_{4}|0)\allowbreak +\allowbreak \frac{(1-a_{2}d)(1-a_{1}a_{4})}{%
(1-a_{1}a_{2}a_{3}a_{4})}a_{3}S_{n-1}^{(3)}(a_{2},a_{3},a_{4}|0)+ \\
\frac{(1-a_{2}a_{4})(1-a_{3}a_{2})}{(1-a_{1}a_{2}a_{3}a_{4})}%
a_{1}S_{n-1}^{(3)}(a_{1},a_{2},a_{4}|0)+ \\
\allowbreak \frac{(1-a_{2}a_{4})(1-a_{2}a_{3})(1-a_{1}a_{4})a_{1}a_{3}}{%
(1-a_{1}a_{2}a_{3}a_{4})}S_{n-2}^{(4)}(a_{1},a_{2},a_{3},a_{4}|0)\allowbreak
,
\end{gather*}%
$\allowbreak $

ii) 
\begin{equation*}
\int_{-1}^{1}g_{5}\left( x|\mathbf{a}^{\left( 5\right) },0\right)
\allowbreak dx=\allowbreak \frac{1-e_{4}\left( \mathbf{a}^{\left( 5\right)
}\right) +e_{5}\left( \mathbf{a}^{\left( 5\right) }\right) e_{1}\left( 
\mathbf{a}^{\left( 5\right) }\right) -e_{5}^{2}\left( \mathbf{a}^{\left(
5\right) }\right) }{\prod_{1\leq j<k\leq 5}(1-a_{j}a_{k})},
\end{equation*}%
where $e_{1},\ldots ,e_{5}$ denote respectively first five elementary
symmetric functions of the vector $\mathbf{a}^{\left( 5\right) }.$ That is $%
\chi _{j}\left( \mathbf{a}^{\left( k\right) }\right) \allowbreak
=\allowbreak \sum_{1\leq n_{1}<n_{2}\ldots <n_{j}\leq
k}\prod_{m=1}^{j}a_{n_{m}}.$
\end{theorem}

\begin{proof}
Is shifted to Section \ref{dowody}.
\end{proof}

For $q\allowbreak =\allowbreak 1$ the problem of finding sequences $%
A_{n}\left( \mathbf{a}^{\left( n\right) }|1\right) $ and $\left\{
T_{j}^{\left( n\right) }\left( \mathbf{a}^{\left( n\right) },1\right)
\right\} _{j\geq 0}$ can be solved completely and trivially. Namely we have:

\begin{proposition}
\label{q=1}%
\begin{eqnarray*}
A_{n}\left( \mathbf{a}^{\left( n\right) }|1\right) &=&\exp \left(
\sum_{1\leq j<k\leq n}a_{j}a_{k}\right) , \\
T_{j}^{\left( n\right) }\left( \mathbf{a}^{\left( n\right) },1\right)
&=&\left( \sum_{k=1}^{n}a_{k}\right) ^{j}.
\end{eqnarray*}
\end{proposition}

\begin{proof}
Using Remark \ref{special} we get: 
\begin{gather*}
g_{n}\left( x|\mathbf{a}^{\left( n\right) },1\right) =\exp \left(
-x^{2}/2+x\sum_{j=1}^{n}a_{j}-\frac{1}{2}\sum_{j=1}^{n}a_{j}^{2}\right) /%
\sqrt{2\pi }\allowbreak \\
\allowbreak =\frac{1}{\sqrt{2\pi }}\exp \left( \frac{1}{2}\left( \left(
\sum_{j=1}^{n}a_{j}\right) ^{2}-\sum_{j=1}^{n}a_{j}^{2}\right) \right) \exp
\left( -x^{2}/2+x\sum_{j=1}^{n}a_{j}-\frac{1}{2}\left(
\sum_{j=1}^{n}a_{j}\right) ^{2}\right) \\
=\exp \left( \sum_{1\leq j<k\leq n}a_{j}a_{k}\right) \frac{\exp \left(
-x^{2}/2\right) }{\sqrt{2\pi }}\sum_{j\geq 0}\frac{\left(
\sum_{k=1}^{n}a_{k}\right) ^{j}}{j!}H_{j}\left( x\right) .
\end{gather*}
\end{proof}

\subsection{Unsolved Problems \& Open Questions}

\label{open}

\subsubsection{Questions}

\begin{itemize}
\item What are the compact forms of functions $\left\{ T_{j}^{\left(
n\right) }\left( \mathbf{a}^{\left( n\right) },q\right) \right\} _{j\geq
0,n\geq 5}$ and $\left\{ A_{n}\left( \mathbf{a}^{\left( n\right) },q\right)
\right\} _{n\geq 5}$ ?

\item What are the compact forms of these functions for $q\allowbreak
=\allowbreak 0$ (free probability case) ?

\item Following formula for $\int_{-1}^{1}g_{5}\left( x|\mathbf{a}^{\left(
5\right) },0\right) \allowbreak dx$ given in assertion ii) of Theorem \ref%
{free} is it true that:%
\begin{equation*}
\int_{-1}^{1}g_{5}\left( x|\mathbf{a}^{\left( 5\right) },q\right)
\allowbreak dx\allowbreak =\allowbreak \frac{\left( \chi _{4}\left( \mathbf{a%
}^{\left( 5\right) }\right) -\chi _{5}\left( \mathbf{a}^{\left( 5\right)
}\right) \chi _{1}\left( \mathbf{a}^{\left( 5\right) }\right) +\chi
_{5}^{2}\left( \mathbf{a}^{\left( 5\right) }\right) \right) _{\infty }}{%
\prod_{1\leq j<k\leq 5}(a_{j}a_{k})_{\infty }}?
\end{equation*}%
Notice that for $a_{5}\allowbreak =\allowbreak 0$ it would reduce to AW
integral.

\item It would be valuable to get values $\left\{ A_{n}\left( \mathbf{a}%
^{\left( n\right) },q\right) \right\} $ for $n\allowbreak =\allowbreak 8,12$
and so on for complex values of parameters $\mathbf{a}^{\left( n\right) }$
but forming conjugate pairs. It would be also fascinating to find
polynomials that would be orthogonalized by so obtained densities.

This problem follows the probabilistic interpretation of Askey--Wilson
density rescaled, with complex parameters. Such interpretation for finite
Markov chains of length at least $3$ was presented in \cite{SzablAW}, \cite%
{Szab-bAW}. Let $\left\{ X_{1},X_{2},X_{3}\right\} $ denote this finite
Markov chain. Then recall that then AW density can be interpreted as the
conditional density of $X_{2}|X_{1},X_{3}$.

It would be exciting to find out if for say $n\allowbreak =\allowbreak 8$
similar probabilistic interpretation could be established. That is if we
could have defined $5-$dimensional random vector $(X_{1},\ldots ,X_{5})$
with normalized function $g_{8}\left( x|\mathbf{a}^{\left( 8\right)
},q\right) $ as the conditional density $X_{3}|X_{1},X_{2},X_{4},X_{5}.$
Note that then the chain $(X_{1},\ldots ,X_{5})$ could not be Markov.

Similar questions apply to the case $n\allowbreak =\allowbreak 12,16,...$ .
\end{itemize}

\subsubsection{Unsolved related problems and direction of further research.}

1. In \cite{Andrews1999} we find Theorem 10.8.2 which is due Gasper and
Rahman (1990) and which can be stated in our notation. For $\max_{1\,\leq
j\leq 5}\left\vert a_{j}\right\vert <1,$ $|q|<1$ we have:%
\begin{equation*}
\int_{-1}^{1}\frac{g_{5}\left( x|\mathbf{a}^{\left( 5\right) },q\right) }{%
\varphi _{h}\left( x|\prod_{j=1}^{5}a_{j},q\right) }dx\allowbreak
=\allowbreak \frac{\prod_{j=1}^{5}\left( \prod_{k=1,k\neq j}^{5}a_{k}\right)
_{\infty }}{\prod_{1\leq j<k\leq 5}\left( a_{j}a_{k}\right) _{\infty }}.
\end{equation*}

This result suggests considering the following functions%
\begin{equation*}
G_{n,m}\left( x|\mathbf{a}^{\left( n\right) },\mathbf{b}^{\left( m\right)
},q\right) \allowbreak =\allowbreak f_{h}\left( x|q\right) \frac{%
\prod_{j=1}^{n}\varphi _{h}\left( x|a_{i},q\right) }{\prod_{k=1}^{m}\varphi
_{h}\left( x|b_{k},q\right) },
\end{equation*}%
where $\mathbf{a}^{\left( n\right) }$ and $\mathbf{b}^{\left( m\right) }$
are certain vectors of dimensions respectively $n$ and $m,$ find its
integrals over $[-1,1]$ and expansions similar to (\ref{symroz}).

2. Recently appeared paper \cite{chung2014} on 'q-Laplace' transform where
many analogies to ordinary case were indicated. What what would be q-Laplace
transform of the distributions that were considered above?.

\section{Proofs}

\label{dowody}

\begin{proof}[Proof of Lemma \protect\ref{C2H}]
We have 
\begin{gather*}
\sum_{k,n,m\geq 0}\frac{a^{n}}{\left( q\right) _{n}}\frac{b^{m}}{\left(
q\right) _{m}}\frac{c^{k}}{\left( q\right) _{k}}h_{n}\left( x|q\right)
h_{m}\left( x|q\right) h_{k}\left( x|q\right) \allowbreak \\
=\frac{1}{\left( ab\right) _{\infty }}\sum_{j\geq 0}\frac{h_{j}\left(
x|q\right) }{\left( q\right) _{j}}\sum_{m=0}^{\infty }\frac{c^{m}}{\left(
q\right) _{m}}\sum_{k=0}^{j}\QATOPD[ ] {j}{k}_{q}c^{k}S_{m+j-k}^{\left(
2\right) }\left( a,b|q\right) .
\end{gather*}%
Since obviously $S_{n}^{\left( 2\right) }\left( a,b|q\right) \allowbreak
=\allowbreak a^{n}w_{n}\left( b/a|q\right) \allowbreak $ we get:$\allowbreak 
$%
\begin{eqnarray*}
&&\sum_{k,n,m\geq 0}\frac{a^{n}}{\left( q\right) _{n}}\frac{b^{m}}{\left(
q\right) _{m}}\frac{c^{k}}{\left( q\right) _{k}}h_{n}\left( x|q\right)
h_{m}\left( x|q\right) h_{k}\left( x|q\right) \\
&=&\frac{1}{\left( ab\right) _{\infty }}\sum_{j\geq 0}\frac{h_{j}\left(
x|q\right) }{\left( q\right) _{j}}\sum_{m=0}^{\infty }\frac{c^{m}}{\left(
q\right) _{m}}\sum_{k=0}^{j}\QATOPD[ ] {j}{k}_{q}c^{k}a^{m+j-k}w_{m+j-k}%
\left( b/a|q\right) \\
&&\frac{1}{\left( ab\right) _{\infty }}\sum_{j\geq 0}\frac{h_{j}\left(
x|q\right) }{\left( q\right) _{j}}\sum_{k=0}^{j}\QATOPD[ ] {j}{k}%
_{q}c^{k}a^{j-k}\sum_{m=0}^{\infty }\frac{(ac)^{m}}{\left( q\right) _{m}}%
w_{m+j-k}\left( b/a|q\right) .
\end{eqnarray*}%
Now we apply formula (\ref{Car11}) and get:$\allowbreak $%
\begin{eqnarray*}
&&\sum_{k,n,m\geq 0}\frac{a^{n}}{\left( q\right) _{n}}\frac{b^{m}}{\left(
q\right) _{m}}\frac{c^{k}}{\left( q\right) _{k}}h_{n}\left( x|q\right)
h_{m}\left( x|q\right) h_{k}\left( x|q\right) \\
&=&\allowbreak \frac{1}{\left( ab,bc,ac\right) _{\infty }}\sum_{j\geq 0}%
\frac{h_{j}\left( x|q\right) }{\left( q\right) _{j}}\sum_{k=0}^{j}\QATOPD[ ]
{j}{k}_{q}c^{k}a^{j-k}\mu _{j-k}\left( b/a|ac,q\right) \frac{1}{\left(
bc\right) _{\infty }\left( bc\right) _{\infty }} \\
&=&\frac{1}{\left( ab,bc,ac\right) _{\infty }}\sum_{j\geq 0}\frac{%
h_{j}\left( x|q\right) }{\left( q\right) _{j}}\sum_{l=0}^{j}\QATOPD[ ] {n}{l}%
_{q}c^{n-l}a^{l}\left( \frac{b}{a}\right) ^{j}\left( \frac{a}{b}\right)
^{j}\mu _{j}\left( \left( \frac{a}{b}\right) ^{-1}|ac,q\right) .
\end{eqnarray*}%
Now we use (\ref{symp1})$\allowbreak \allowbreak \allowbreak $and
Proposition \ref{ch_sym}, ii) and get:%
\begin{eqnarray*}
&&\sum_{k,n,m\geq 0}\frac{a^{n}}{\left( q\right) _{n}}\frac{b^{m}}{\left(
q\right) _{m}}\frac{c^{k}}{\left( q\right) _{k}}h_{n}\left( x|q\right)
h_{m}\left( x|q\right) h_{k}\left( x|q\right) \\
&=&\frac{1}{\left( ab,bc,ac\right) _{\infty }}\sum_{j\geq 0}\frac{%
h_{j}\left( x|q\right) }{\left( q\right) _{j}}\sum_{l=0}^{j}\QATOPD[ ] {j}{l}%
_{q}c^{j-l}b^{l}\sum_{k=0}^{l}\QATOPD[ ] {l}{k}_{q}\left( -ac\right) ^{k}q^{%
\binom{k}{2}}w_{l-k}\left( \frac{a}{b}|q\right) \\
&=&\frac{1}{\left( ab,bc,ac\right) _{\infty }}\sum_{j\geq 0}\frac{%
h_{j}\left( x|q\right) }{\left( q\right) _{j}}\sum_{k=0}^{j}\QATOPD[ ] {j}{k}%
_{q}\left( -ac\right) ^{k}q^{\binom{k}{2}}\sum_{l=k}^{j}\QATOPD[ ] {j-k}{l-k}%
_{q}c^{j-l}b^{l}w_{l-k}\left( \frac{a}{b}|q\right) \\
&=&\frac{1}{\left( ab,bc,ac\right) _{\infty }}\sum_{j\geq 0}\frac{%
h_{j}\left( x|q\right) }{\left( q\right) _{j}}\sum_{k=0}^{j}\QATOPD[ ] {j}{k}%
_{q}\left( -ac\right) ^{k}q^{\binom{k}{2}}\sum_{m=0}^{j-k}\QATOPD[ ] {j-k}{m}%
_{q}c^{j-k-m}b^{k+m}w_{m}\left( a/b|q\right) \\
&=&\frac{1}{\left( ab,bc,ac\right) _{\infty }}\sum_{j\geq 0}\frac{%
h_{j}\left( x|q\right) }{\left( q\right) _{j}}\sum_{k=0}^{j}\QATOPD[ ] {j}{k}%
_{q}\left( -abc\right) ^{k}q^{\binom{k}{2}}\sum_{m=0}^{j-k}\QATOPD[ ] {j-k}{m%
}_{q}c^{j-k-m}S_{m}^{\left( 2\right) }\left( a,b|q\right) .
\end{eqnarray*}
\end{proof}

\begin{proof}[Proof of Theorem \protect\ref{fAW}]
Applying (\ref{rozfQ}) we get: 
\begin{eqnarray*}
&&\sum_{k,n,m,j\geq 0}\frac{a^{n}}{\left( q\right) _{n}}\frac{b^{m}}{\left(
q\right) _{m}}\frac{c^{k}}{\left( q\right) _{k}}\frac{d^{j}}{\left( q\right)
_{j}}h_{n}\left( x|q\right) h_{m}\left( x|q\right) h_{k}\left( x|q\right)
h_{j}\left( x|q\right) \allowbreak \\
&=&\frac{1}{\left( ab,cd\right) _{\infty }}\sum_{m,k\geq 0}\frac{%
S_{m}^{(2)}\left( a,b\right) S_{k}^{(2)}\left( c,d\right) }{\left( q\right)
_{m}\left( q\right) _{k}}h_{m}\left( x|q\right) h_{k}\left( x|q\right) \\
&=&\frac{1}{\left( ab,cd\right) _{\infty }}\sum_{m,k\geq 0}\frac{%
S_{m}^{(2)}\left( a,b\right) S_{k}^{(2)}\left( c,d\right) }{\left( q\right)
_{m}\left( q\right) _{k}}\sum_{j=0}^{\min (m,k)}\QATOPD[ ] {m}{j}_{q}\QATOPD[
] {k}{j}_{q}\left( q\right) _{j}h_{m+k-2j}\left( x|q\right) \allowbreak \\
&=&\allowbreak \frac{1}{\left( ab,cd\right) _{\infty }}\sum_{j\geq 0}\frac{%
(ac)^{j}}{\left( q\right) _{j}}\sum_{m,k\geq j}\frac{a^{m-j}c^{k-j}w_{m}%
\left( b/a|q\right) w_{k}\left( d/c|q\right) }{\left( q\right) _{m-j}\left(
q\right) _{k-j}}h_{m-j+k-j}\left( x|q\right)
\end{eqnarray*}%
and further 
\begin{eqnarray*}
&&\sum_{k,n,m,j\geq 0}\frac{a^{n}}{\left( q\right) _{n}}\frac{b^{m}}{\left(
q\right) _{m}}\frac{c^{k}}{\left( q\right) _{k}}\frac{d^{j}}{\left( q\right)
_{j}}h_{n}\left( x|q\right) h_{m}\left( x|q\right) h_{k}\left( x|q\right)
h_{j}\left( x|q\right) \\
&=&\allowbreak \frac{1}{\left( ab,cd\right) _{\infty }}\sum_{j\geq 0}\frac{%
(ac)^{j}}{\left( q\right) _{j}}\sum_{s,t\geq 0}\frac{a^{s}c^{t}w_{s+j}\left(
b/a|q\right) w_{t+j}\left( d/c|q\right) }{\left( q\right) _{s}\left(
q\right) _{t}}h_{s+t}\left( x|q\right) \allowbreak \\
&=&\allowbreak \frac{1}{\left( ab,cd\right) _{\infty }}\sum_{j\geq 0}\frac{%
(ac)^{j}}{\left( q\right) _{j}}\sum_{n\geq 0}\frac{h_{n}\left( x|q\right) }{%
\left( q\right) _{n}}\sum_{k=0}^{n}\QATOPD[ ] {n}{k}_{q}a^{k}c^{n-k}w_{k+j}%
\left( b/a|q\right) w_{j+n-k}\left( d/c|q\right) \\
&=&\allowbreak \frac{1}{\left( ab,cd\right) _{\infty }}\sum_{n\geq 0}\frac{%
h_{n}\left( x|q\right) }{\left( q\right) _{n}}\sum_{k=0}^{n}\QATOPD[ ] {n}{k}%
_{q}a^{k}c^{n-k}\sum_{j\geq 0}\frac{(ac)^{j}}{\left( q\right) _{j}}%
w_{k+j}\left( b/a|q\right) w_{j+n-k}\left( d/c|q\right)
\end{eqnarray*}%
Now we apply Carlitz formulae (\ref{Car21}) and (\ref{Car22}) getting:%
\begin{eqnarray*}
&&\sum_{k,n,m,j\geq 0}\frac{a^{n}}{\left( q\right) _{n}}\frac{b^{m}}{\left(
q\right) _{m}}\frac{c^{k}}{\left( q\right) _{k}}\frac{d^{j}}{\left( q\right)
_{j}}h_{n}\left( x|q\right) h_{m}\left( x|q\right) h_{k}\left( x|q\right)
h_{j}\left( x|q\right) \\
&=&\frac{\left( abcd\right) _{\infty }}{\left( ab,cd,ac,bc,ad,bd\right)
_{\infty }}\sum_{n\geq 0}\frac{h_{n}\left( x|q\right) }{\left( q\right) _{n}}%
\sum_{k=0}^{n}\QATOPD[ ] {n}{k}_{q}a^{k}c^{n-k}\times \\
&&\sum_{s=0}^{k}\sum_{t=0}^{n-k}\QATOPD[ ] {k}{s}_{q}\QATOPD[ ] {n-k}{t}_{q}%
\frac{\left( cb\right) _{s}\left( ad\right) _{t}\left( bd\right) _{s+t}}{%
\left( abcd\right) _{s+t}}\left( \frac{b}{a}\right) ^{k-s}\left( \frac{d}{c}%
\right) ^{n-k-t} \\
&=&\frac{\left( abcd\right) _{\infty }}{\left( ab,cd,ac,bc,ad,bd\right)
_{\infty }}\sum_{n\geq 0}\frac{h_{n}\left( x|q\right) }{\left( q\right) _{n}}%
\sum_{k=0}^{n}\QATOPD[ ] {n}{k}_{q}\times \\
&&\sum_{s=0}^{k}\sum_{t=0}^{n-k}\QATOPD[ ] {k}{s}_{q}\QATOPD[ ] {n-k}{t}_{q}%
\frac{\left( cb\right) _{s}\left( ad\right) _{t}\left( bd\right) _{s+t}}{%
\left( abcd\right) _{s+t}}a^{s}b^{k-s}c^{t}d^{n-k-t}.
\end{eqnarray*}%
Thus it remains to show that for every $n\geq 0.$

\begin{eqnarray*}
&&\sum_{k=0}^{n}\QATOPD[ ] {n}{k}_{q}\sum_{s=0}^{k}\sum_{t=0}^{n-k}\QATOPD[ ]
{k}{s}_{q}\QATOPD[ ] {n-k}{t}_{q}\frac{\left( cb\right) _{s}\left( ad\right)
_{t}\left( bd\right) _{s+t}}{\left( abcd\right) _{s+t}}%
a^{s}b^{k-s}c^{t}d^{n-k-t} \\
&=&\sum_{j=0}^{n}\QATOPD[ ] {n}{j}_{q}\frac{\left( bd\right) _{j}}{\left(
abcd\right) _{j}}S_{n-j}^{\left( 2\right) }\left( b,d|q\right) \sum_{k=0}^{j}%
\QATOPD[ ] {j}{k}_{q}\left( cb\right) _{k}a^{k}\left( ad\right)
_{j-k}c^{j-k}.
\end{eqnarray*}%
This fact follows the following calculations:%
\begin{eqnarray*}
&&\sum_{k=0}^{n}\QATOPD[ ] {n}{k}_{q}\sum_{s=0}^{k}\sum_{t=0}^{n-k}\QATOPD[ ]
{k}{s}_{q}\QATOPD[ ] {n-k}{t}_{q}\frac{\left( cb\right) _{s}\left( ad\right)
_{t}\left( bd\right) _{s+t}}{\left( abcd\right) _{s+t}}%
a^{s}b^{k-s}c^{t}d^{n-k-t} \\
&=&\sum_{s,t\geq 0,s+t\leq n}\frac{\left( q\right) _{n}}{\left( q\right)
_{s}\left( q\right) _{t}\left( q\right) _{n-s-t}}\frac{\left( cb\right)
_{s}\left( ad\right) _{t}\left( bd\right) _{s+t}}{\left( abcd\right) _{s+t}}%
a^{s}c^{t}\sum_{k=s\vee n-t}^{n}\QATOPD[ ] {n-t-s}{k-s}_{q}b^{k-s}d^{n-k-t}%
\allowbreak \\
&=&\sum_{s,t\geq 0,s+t\leq n}\frac{\left( q\right) _{n}}{\left( q\right)
_{s}\left( q\right) _{t}\left( q\right) _{n-s-t}}\frac{\left( cb\right)
_{s}\left( ad\right) _{t}\left( bd\right) _{s+t}}{\left( abcd\right) _{s+t}}%
a^{s}c^{t}\sum_{m=0\vee n-t-s}^{n-s}\QATOPD[ ] {n-t-s}{m}_{q}b^{m}d^{n-s-m-t}
\\
&=&\sum_{s,t\geq 0,s+t\leq n}\frac{\left( q\right) _{n}}{\left( q\right)
_{s}\left( q\right) _{t}\left( q\right) _{n-s-t}}\frac{\left( cb\right)
_{s}\left( ad\right) _{t}\left( bd\right) _{s+t}}{\left( abcd\right) _{s+t}}%
a^{s}c^{t}S_{n-t-s}^{\left( 2\right) }\left( b,d|q\right) .
\end{eqnarray*}%
Now we introduce new indices of summation: $j=t+s,\allowbreak k=s.$ We have
then%
\begin{eqnarray*}
&&\sum_{s,t\geq 0,s+t\leq n}\frac{\left( q\right) _{n}}{\left( q\right)
_{s}\left( q\right) _{t}\left( q\right) _{n-s-t}}\frac{\left( cb\right)
_{s}\left( ad\right) _{t}\left( bd\right) _{s+t}}{\left( abcd\right) _{s+t}}%
a^{s}c^{t}S_{n-t-s}^{\left( 2\right) }\left( b,d|q\right) \\
&=&\sum_{j=0}^{n}\QATOPD[ ] {n}{j}_{q}\frac{\left( bd\right) _{j}}{\left(
abcd\right) _{j}}S_{n-j}^{\left( 2\right) }\left( b,d|q\right) \sum_{k=0}^{j}%
\QATOPD[ ] {j}{k}_{q}\left( cb\right) _{k}a^{k}\left( ad\right)
_{j-k}c^{j-k}.
\end{eqnarray*}
\end{proof}

\begin{proof}[Proof of Proposition \protect\ref{GEN}]
Notice that for $n\allowbreak =\allowbreak 0$ our formulae are true since we
have: $g_{1}\left( x|a_{1},q\right) \allowbreak =\allowbreak f_{h}\left(
x|q\right) \varphi _{h}\left( x|a_{1},q\right) \allowbreak =\allowbreak
f_{h}\left( x|q\right) \sum_{m\geq 0}\frac{a_{1}^{m}}{\left( q\right) _{m}}%
h_{m}\left( x|q\right) ,$ So $T_{m}^{\left( 1\right) }\left( a_{1},q\right)
\allowbreak =\allowbreak a_{1}^{m}$ and $A_{1}\left( a_{1},q\right)
\allowbreak =\allowbreak 1.$ Next notice that:%
\begin{equation*}
g_{n+1}\left( x|\mathbf{a}^{\left( n+1\right) },q\right) =g_{n}\left( x|%
\mathbf{a}^{\left( n\right) },q\right) \varphi _{h}\left( x|a_{n+1},q\right)
,
\end{equation*}%
where we understand $\mathbf{a}^{\left( n+1\right) }\allowbreak =\allowbreak
\left( a_{1},\ldots ,a_{n},a_{n+1}\right) .$ So by induction assumption the
left hand side of (\ref{symroz}) is equal to:%
\begin{equation*}
A_{n+1}\left( \mathbf{a}^{\left( n+1\right) },q\right) f_{h}\left(
x|q\right) \sum_{j\geq 0}\frac{T_{j}^{\left( n+1\right) }\left( \mathbf{a}%
^{\left( n+1\right) },q\right) }{\left( q\right) _{j}}h_{j}\left( x|q\right)
,
\end{equation*}%
while the right hand side to 
\begin{equation*}
A_{n}\left( \mathbf{a}^{\left( n\right) },q\right) f_{h}\left( x|q\right)
\sum_{j,k\geq 0}\frac{T_{j}^{\left( n\right) }\left( \mathbf{a}^{\left(
n\right) },q\right) a_{n+1}^{k}}{\left( q\right) _{j}\left( q\right) _{k}}%
h_{j}\left( x|q\right) h_{k}\left( x|q\right) .
\end{equation*}%
We apply again (\ref{linH}) getting:%
\begin{eqnarray*}
&&\sum_{j,k\geq 0}\frac{T_{j}^{\left( n\right) }\left( \mathbf{a}^{\left(
n\right) },q\right) a_{n+1}^{k}}{\left( q\right) _{j}\left( q\right) _{k}}%
h_{j}\left( x|q\right) h_{k}\left( x|q\right) \\
&=&\sum_{j,k\geq 0}\frac{T_{j}^{\left( n\right) }\left( \mathbf{a}^{\left(
n\right) },q\right) a_{n+1}^{k}}{\left( q\right) _{j}\left( q\right) _{k}}%
\sum_{m=0}^{j\wedge k}\QATOPD[ ] {k}{m}_{q}\QATOPD[ ] {j}{m}_{q}\left(
q\right) _{m}h_{j+k-2m}\left( x|q\right) \\
&=&\sum_{m\geq 0}\frac{a_{n+1}^{m}}{\left( q\right) _{m}}\sum_{k,j\geq m}%
\frac{a_{n+1}^{k-m}T_{j}^{\left( n\right) }\left( \mathbf{a}^{\left(
n\right) },q\right) }{\left( q\right) _{k-m}\left( q\right) _{j-m}}%
h_{j+k-2m}\left( x|q\right) \\
&=&\sum_{m\geq 0}\frac{a_{n+1}^{m}}{\left( q\right) _{m}}\sum_{s,t\geq 0}%
\frac{a_{n+1}^{s}T_{t+m}^{\left( n\right) }\left( \mathbf{a}^{\left(
n\right) },q\right) }{\left( q\right) _{s}\left( q\right) _{t}}h_{s+t}\left(
x|q\right)
\end{eqnarray*}

Next we introduce new indices of summation $r\allowbreak =\allowbreak s+t$
and $j\allowbreak =\allowbreak s$ and get: 
\begin{gather*}
\sum_{j,k\geq 0}\frac{T_{j}^{\left( n\right) }\left( \mathbf{a}^{\left(
n\right) },q\right) a_{n+1}^{k}}{\left( q\right) _{j}\left( q\right) _{k}}%
h_{j}\left( x|q\right) h_{k}\left( x|q\right) =\sum_{m\geq 0}\frac{%
a_{n+1}^{m}}{\left( q\right) _{m}}\sum_{r=0}^{\infty }\frac{h_{r}\left(
x|q\right) }{\left( q\right) _{r}}\sum_{j=0}^{r}\QATOPD[ ] {r}{j}%
_{q}a_{n+1}^{j}T_{m+r-j}^{\left( n\right) }\left( \mathbf{a}^{\left(
n\right) },q\right) \\
=\sum_{r=0}^{\infty }\frac{h_{r}\left( x|q\right) }{\left( q\right) _{r}}%
\sum_{j=0}^{r}\QATOPD[ ] {r}{j}_{q}a_{n+1}^{j}\sum_{m\geq 0}\frac{a_{n+1}^{m}%
}{\left( q\right) _{m}}T_{m+r-j}^{\left( n\right) }\left( \mathbf{a}^{\left(
n\right) },q\right) \\
=\sum_{m\geq 0}\frac{a_{n+1}^{m}}{\left( q\right) _{m}}T_{m}^{\left(
n\right) }\left( \mathbf{a}^{\left( n\right) },q\right) \sum_{r=0}^{\infty }%
\frac{h_{r}\left( x|q\right) }{\left( q\right) _{r}}\sum_{j=0}^{r}\QATOPD[ ]
{r}{j}_{q}a_{n+1}^{j}H_{r-j}^{\left( n\right) }\left( \mathbf{a}^{\left(
n\right) },q\right) \\
=\frac{A_{n+1}\left( \mathbf{a}^{\left( n+1\right) },q\right) }{A_{n}\left( 
\mathbf{a}^{\left( n\right) },q\right) }\sum_{r=0}^{\infty }\frac{%
h_{r}\left( x|q\right) }{\left( q\right) _{r}}\sum_{j=0}^{r}\QATOPD[ ] {r}{j}%
_{q}a_{n+1}^{j}H_{r-j}^{\left( n\right) }\left( \mathbf{a}^{\left( n\right)
},q\right) .
\end{gather*}
\end{proof}

\begin{proof}[Proof of Theorem \protect\ref{free}]
We use (\ref{sig4}) and utilizing Remark \ref{particular} we get:%
\begin{gather*}
\sigma _{n}^{\left( 4\right) }\left( a,b,c,d|0\right) \allowbreak
=\allowbreak S_{n}^{\left( 2\right) }(b,d|0)\allowbreak +\allowbreak \frac{%
(1-bd)}{(1-abcd)}\sum_{j=1}^{n}S_{n-j}^{\left( 2\right) }\left( b,d|q\right)
((1-ad)c^{j}\allowbreak + \\
\allowbreak (1-cb)a^{j}\allowbreak +\allowbreak
(1-cb)(1-ad)ac\sum_{k=1}^{j-1}a^{k-1}c^{j-1-k}).\allowbreak 
\end{gather*}%
And further%
\begin{gather*}
\sigma _{n}^{\left( 4\right) }\left( a,b,c,d|0\right) =\allowbreak
S_{n}^{\left( 2\right) }(b,d|0)\allowbreak +\allowbreak \frac{(1-bd)}{%
(1-abcd)}\sum_{j=1}^{n}S_{n-j}^{\left( 2\right) }\left( b,d|q\right)
((1-ad)c^{j}\allowbreak +\allowbreak (1-cb)a^{j}\allowbreak +\allowbreak  \\
(1-cb)(1-ad)acS_{j-2}^{\left( 2\right) }(a,c|0)\allowbreak \allowbreak  \\
=S_{n}^{\left( 2\right) }(b,d|0)\allowbreak +\allowbreak \frac{(1-bd)}{%
(1-abcd)}\sum_{j=1}^{n}S_{n-j}^{\left( 2\right) }\left( b,d|q\right)
((1-ad)c^{j}\allowbreak +\allowbreak (1-cb)a^{j})\allowbreak  \\
+\allowbreak \frac{(1-bd)(1-cb)(1-ad)ac}{(1-abcd)}\sum_{j=2}^{n}S_{n-j}^{%
\left( 2\right) }\left( b,d|q\right) S_{j-2}^{\left( 2\right)
}(a,c|0)\allowbreak .
\end{gather*}%
Now we use formula (\ref{rozklS}). Then we replace $a$ by $a_{1}$ , $b$ by $%
a_{2}$ and so on. Finally we use formulae (\ref{_5}) and (\ref{chS}) which
remembering that $\left( 0\right) _{n}\allowbreak =\allowbreak 1$ leads to
our integral formula.
\end{proof}

\begin{acknowledgement}
The author is very grateful to an unknown referee for pointing out
additional references and just evaluation of the paper.
\end{acknowledgement}


\begin{thebibliography}{99}
\bibitem{Andrews1999} Andrews, George E.; Askey, Richard; Roy, Ranjan.
Special functions. Encyclopedia of Mathematics and its Applications, 71. 
\emph{Cambridge University Press,} Cambridge, 1999. xvi+664 pp. ISBN:
0-521-62321-9; 0-521-78988-5 MR1688958 (2000g:33001) (2002k:33011)

\bibitem{bressoud} Bressoud, D. M. A simple proof of Mehler's formula for
\$q\$-Hermite polynomials. \emph{Indiana Univ. Math. J.} 29 (1980), no. 
\textbf{4}, 577--580. MR0578207 (81f:33009)

\bibitem{bryc1} Bryc, W\l odzimierz. Stationary random fields with linear
regressions. \emph{Ann. Probab.} \textbf{29} (2001), no. 1, 504--519.
MR1825162 (2002d:60014)

\bibitem{bms} Bryc, W\l odzimierz; Matysiak, Wojciech; Szab\l owski, Pawe\l\ %
J. Probabilistic aspects of Al-Salam-Chihara polynomials. \emph{Proc. Amer.
Math. Soc.} \textbf{133} (2005), no. 4, 1127--1134 (electronic). MR2117214
(2005m:33033)

\bibitem{Corteel10} Corteel, Sylvie; Williams, Lauren K. Staircase tableaux,
the asymmetric exclusion process, and Askey-Wilson polynomials. Proc. Natl.
Acad. Sci. USA 107 (2010), no. 15, 6726--6730. MR2630104

\bibitem{Carlitz72} Carlitz, L. Generating functions for certain
\$Q\$-orthogonal polynomials. \emph{Collect. Math.} \textbf{23} (1972),
91--104. MR0316773 (47 \#5321)

\bibitem{Carlitz56} Carlitz, L. Some polynomials related to theta functions.
Ann. Mat. Pura Appl. (4) 41 (1956), 359--373. MR0078510 (17,1205e)

\bibitem{Carlitz57} Carlitz, L. Some polynomials related to Theta functions. 
\emph{Duke Math. J.} \textbf{24} (1957), 521--527. MR0090672 (19,849e)

\bibitem{chung2014} Chung, Won Sang; Kim, Taekyun; Kwon, Hyuck In. On the
\$q\$-analog of the Laplace transform. \emph{Russ. J. Math. Phys. }21
(2014), no. \textbf{2}, 156--168. MR3215667

\bibitem{IA} Ismail, Mourad E. H. Classical and quantum orthogonal
polynomials in one variable. With two chapters by Walter Van Assche. With a
foreword by Richard A. Askey. Encyclopedia of Mathematics and its
Applications, 98. \emph{Cambridge University Press,} Cambridge, 2005.
xviii+706 pp. ISBN: 978-0-521-78201-2; 0-521-78201-5 MR2191786 (2007f:33001)

\bibitem{ISV87} Ismail, Mourad E. H.; Stanton, Dennis; Viennot, G\'{e}rard.
The combinatorics of \$q\$-Hermite polynomials and the Askey-Wilson
integral. \emph{European J. Combin.} \textbf{8} (1987), no. 4, 379--392.
MR0930175 (89h:33015)

\bibitem{KLS} Koekoek, Roelof; Lesky, Peter A.; Swarttouw, Ren\'{e} F.
Hypergeometric orthogonal polynomials and their \$q\$-analogues. With a
foreword by Tom H. Koornwinder. \emph{Springer Monographs in Mathematics.
Springer-Verlag}, Berlin, 2010. xx+578 pp. ISBN: 978-3-642-05013-8 MR2656096
(2011e:33029)

\bibitem{Liu2009} Liu, Zhi-Guo. An identity of Andrews and the Askey-Wilson
integral. \emph{Ramanujan J.} 19 (2009), no. 1, 115--119. MR2501242
(2010g:33007)

\bibitem{Ma2011} Ma, X. R. A new proof of the Askey-Wilson integral via a
five-variable Ramanujan's reciprocity theorem. \emph{Ramanujan J.} 24
(2011), no. 1, 61--65. MR2765601

\bibitem{Szab-Exp} Szab\l owski, Pawe\l\ J. Expansions of one density via
polynomials orthogonal with respect to the other. \emph{J. Math. Anal. Appl.}
\textbf{383} (2011), no. 1, 35--54. MR2812716, http://arxiv.org/abs/1011.1492

\bibitem{SzablAW} Szab\l owski, Pawe\l\ J. On the structure and
probabilistic interpretation of Askey-Wilson densities and polynomials with
complex parameters. \emph{J. Funct. Anal.} \textbf{261} (2011), no. 3,
635--659. MR2799574, http://arxiv.org/abs/1011.1541

\bibitem{Szab-bAW} Szab\l owski, Pawe\l\ J. Befriending Askey--Wilson
polynomials, submitted, in print in \emph{Infin. Dimens. Anal. Quantum
Probab. Relat. Top.} \emph{, }http://arxiv.org/abs/1111.0601.

\bibitem{Szabl-peculiar} Szab\l owski, Pawe\l\ J., On peculiar properties of
generating functions of some orthogonal polynomials, \emph{J. Phys. A: Math.
Theor.} \textbf{45} (2012) 365207 (12pp) http://arxiv.org/abs/1204.0972

\bibitem{Szab-rev} Szab\l owski, Pawe\l , J. On the $q-$Hermite polynomials
and their relationship with some other families of orthogonal polynomials, 
\emph{Dem. Math.} \textbf{66(}2013) no.4, 679-708,
http://arxiv.org/abs/1101.2875,
\end{thebibliography}
\end{document}